\documentclass[12pt,twoside]{amsart}

\usepackage{amssymb,latexsym}

\usepackage{hyperref}
\hypersetup{
  colorlinks   = true, 
  urlcolor     = blue, 
  linkcolor    = blue, 
  citecolor   = red 
}

\ifpdf
\usepackage[pdftex]{graphicx}
\else
\usepackage{graphicx}
\fi

\usepackage{anysize}

\newtheorem{theorem}{Theorem}[section]

\newtheorem{lemma}[theorem]{Lemma}

\newtheorem{conjecture}[theorem]{Conjecture}

\theoremstyle{definition}

\theoremstyle{remark}
\newtheorem{remark}[theorem]{Remark}

\numberwithin{equation}{section}

\DeclareMathOperator{\vol}{vol}

\newcommand{\dist}{\mathop{\rm dist}}

\renewcommand{\epsilon}{\varepsilon}
\renewcommand{\phi}{\varphi}
\renewcommand{\kappa}{\varkappa}

\begin{document}

\title{Optimality of codes with respect to error probability in Gaussian noise}

\author{Alexey~Balitskiy{$^{\spadesuit}$}}

\email{alexey\_m39@mail.ru}

\author{Roman~Karasev{$^{\clubsuit}$}}

\email{r\_n\_karasev@mail.ru}
\urladdr{http://www.rkarasev.ru/en/}

\author{Alexander~Tsigler{$^{\diamondsuit}$}}

\email{sasha-cigler@mail.ru}

\thanks{{$^{\spadesuit}$}{$^{\clubsuit}$}{$^{\diamondsuit}$}Supported by the Russian Foundation for Basic Research grant 15-31-20403 (mol\_a\_ved)}

\thanks{{$^{\spadesuit}$}{$^{\clubsuit}$}{$^{\diamondsuit}$}Supported by the Russian Foundation for Basic Research grant 15-01-99563 A}

\address{{$^{\spadesuit}$}{$^{\clubsuit}$}{$^{\diamondsuit}$}Moscow Institute of Physics and Technology, Institutskiy per. 9, Dolgoprudny, Russia 141700}

\address{{$^{\spadesuit}$}{$^{\clubsuit}$}{$^{\diamondsuit}$}Institute for Information Transmission Problems RAS, Bolshoy Karetny per. 19, Moscow, Russia 127994}

\begin{abstract}
We consider geometrical optimization problems related to optimizing the error probability in the presence of a Gaussian noise. One famous questions in the field is the ``weak simplex conjecture''. We discuss possible approaches to it, and state related conjectures about the Gaussian measure, in particular, the conjecture about minimizing of the Gaussian measure of a simplex. We also consider antipodal codes, apply the \v{S}id\'ak inequality and establish some theoretical and some numerical results about their optimality.
\end{abstract}

\maketitle

\section{Introduction}

Assume a code is represented by a finite set of vectors $\{v_i\}$ in $\mathbb R^n$ and the decoding procedure is by taking the (Euclidean distance) closest point of $\{v_i\}$ (which is optimal subject to likelihood maximization).

If we want to calculate the probability that a vector is transmitted correctly in the presence of a normalized Gaussian noise then we obtain a value proportional to
\begin{equation}
\label{equation:funcional}
P(v_1,\ldots, v_N) = \int_{\mathbb R^n} \max_i e^{-|x - v_i|^2}\; dx.
\end{equation}

\begin{remark}
If the probability to choose any of $v_i$ in the code is the same then the actual probability of transmitting the signal correctly is the above value $P$ multiplied by a constant and divided by $N$. In most if this text the constant in front of the integral is not relevant, but if someone want to interpret the practical meaning of the data given in Section~\ref{section:antipodal-numerical} (where we allow $N$ to vary) then this factor has to be taken into account.
\end{remark}

\begin{remark} Here we normalize the exponent as $e^{-x^2}$ and do not use the leading factor $(\pi)^{-n/2}$ to shorten the formulas. Again, in the numerical results of Section~\ref{section:antipodal-numerical} we will use the more common normalization with density $(2\pi)^{-n/2} e^{-|x|^2/2}$.
\end{remark}

It was conjectured that in the case $N=n+1$ with fixed total energy $|v_1|^2+\dots + |v_N|^2$ the maximum of this functional is attained at regular simplices centered at the origin, see~\cite[Page~74]{cover1987}. In~\cite{dunbridge1965,dunbridge1967} it was shown that the regular simplex is optimal for energy tending to infinity and locally optimal for every energy. Eventually, this conjecture turned out to be false; in~\cite{stein1994} it was shown that for $m\ge 7$ the configuration with $m-2$ zero vectors and $2$ antipodal vectors is better than the regular simplex with $m$ vertices and the same energy. Now it is conjectured that

\begin{conjecture}[The simplex conjecture]
\label{conjecture:simplex}
For $N=n+1$ and fixed $|v_1|=\dots = |v_N|=r$ the maximum of $P(v_1,\ldots, v_N)$ is attained at any configuration forming a regular simplex inscribed into the ball of radius $r$.
\end{conjecture}

The case of energy tending to zero or to infinity in this conjecture was considered in~\cite{balakri1961}; its validity for $n=3$ was established in~\cite{ls1966}.

Our plan is as follows: In Sections \ref{section:slicing-sphere} and \ref{section:slicing-gaussian} we explain how the problems of maximizing $P$ may be reduced to problems of optimal covering of a sphere by caps and to minimizing Gaussian measure of an outscribed simplex. In Theorem~\ref{theorem:sidak} of Section~\ref{section:slicing-gaussian} and in Section~\ref{section:antipodal} we prove some optimality results for antipodal configurations, inspired by such a configuration in the example of Steiner~\cite{stein1994}. In Section~\ref{section:antipodal-numerical} we give some numerical results for antipodal configurations.

The paper is organized as follows: In Section~\ref{section:slicing-sphere} we overview the techniques that prove Conjecture~\ref{conjecture:simplex} in dimensions $\le 3$. In Section~\ref{section:slicing-gaussian} we provide another approach that would reduce the problem to another conjecture about the Gaussian measure of a (generalized) simplex (Conjecture~\ref{conjecture:gausssimplex}) and we show that this approach does give some information for ``antipodal'' configurations of points, where, for every point $x$ present in the configuration, the point $-x$ is also present, here we recall and use \v{S}id\'{a}k's lemma about the Gaussian measure. In Section~\ref{section:antipodal-numerical} we study the optimality of antipodal configurations with varying lengths of vectors keeping the total energy, establish the optimality of the equal length configuration for $4$-point antipodal configurations, and make numeric tests showing that for larger number of vectors the optimal lengths have more complex behavior.

\subsection*{Acknowledgments.}
The authors thank Grigori Kabatianski for explaining this problem to us.

\section{Slicing with the uniform measure}
\label{section:slicing-sphere}

The first thing that comes to mind is to represent the integral as ``the volume under the graph'', that is
\begin{equation}
\label{equation:volume}
P(v_1,\ldots, v_N) = \vol \left\{(x, y)\in\mathbb R^{n+1} : \exists i\ 0\le y \le e^{-|x-v_i|^2}\right\}.
\end{equation}

Then we can fix a value $y\in [0,1]$ and try to maximize the $n$-dimensional volume of the corresponding slice of the set in the right hand side of (\ref{equation:volume}). If the maximum of the volume of the section will be obtained at the same configuration for every value of $y$ then the maximum of the total volume $P(v_1,\ldots, v_N)$ will also be there.

The corresponding slice is the set $\bigcup_i \{ |x - v_i|^2 \le -\ln y \}$, that is a union of balls of the same radius. So Conjecture~\ref{conjecture:simplex} would follow from the following stronger

\begin{conjecture}
\label{conjecture:balls}
For $N=n+1$ and fixed $|v_1|=\dots = |v_N|=r$ and $R>0$ the maximal volume of the union of balls $\bigcup_i B_{v_i} (R)$ is attained at any configuration forming a regular simplex inscribed into the ball of radius $r$.
\end{conjecture}

By further slicing with the distance to the origin this conjecture would follow from an even stronger

\begin{conjecture}
\label{conjecture:caps}
For $N=n+1$ and $R>0$ the maximal area of the union of $N$ spherical caps of radius $R$ in the unit sphere $\mathbb S^{n-1}$ is attained when the centers of the caps form a regular simplex inscribed into the unit sphere.
\end{conjecture}

In fact, the case $n=3$ of the latter conjecture (when the sphere is $2$-dimensional) was resolved positively in~\cite{lft1953}, this was noted in~\cite{ls1966} and resulted in

\begin{theorem}[Landau--Slepian, 1966]
\label{theorem:fejestoth}
Conjecture~\ref{conjecture:simplex} holds true for $n=3$ and $N=4$.
\end{theorem}

Moreover, in~\cite{lft1953} other regular configurations (corresponding to the vertices of a regular solid body) were proved to maximize the area of the union of equal caps, resulting in optimality of the corresponding spherical codes. This was also noted in~\cite{ls1966}.

After that, in~\cite{ls1966} two analytical-geometrical lemmas about the caps on a two-dimensional sphere were shown to hold in larger dimensions and it was concluded that Conjecture~\ref{conjecture:simplex} was therefore established for arbitrary $n$. However, the proof of Conjecture~\ref{conjecture:caps} for $n=3$ in~\cite{lft1953} does not generalize to larger dimensions because the argument only work for the case when in the presumably optimal configuration the caps only intersect pairwise and no point is covered by three of them. This is not a problem in $\mathbb S^2$, since when three caps intersect in the regular configuration then those caps cover the whole $\mathbb S^2$ and the assertion holds trivially.

In the thesis~\cite{farber1968} we see that this problem in the argument of~\cite{ls1966} was evident to the experts.

\section{Slicing with the Gaussian measure}
\label{section:slicing-gaussian}

Here we propose a different approach reducing the problem to estimates for Gaussian measures instead of spherical measures. Let us rewrite the value to optimize differently:

\begin{equation}
\label{equation:gaussian}
P(v_1,\ldots, v_N) = \int_{\mathbb R^n} \max_i e^{2 x\cdot v_i  - |v_i|^2}\; e^{-|x|^2}dx.
= \bar \mu \left\{(x, y)\in\mathbb R^{n+1} : \exists i\ 0\le y \le e^{2 x\cdot v_i  - |v_i|^2}\right\},
\end{equation}
here $\bar \mu$ is the measure with density $e^{-|x|^2} dx dy$. Again, we can fix $y$ now and maximize the measure $\mu$ of any section, where $\mu$ is the Gaussian measure with density $e^{-|x|^2}dx$. The set whose measure is maximized will be a union of halfspaces:
$$
C_y(v_1,\ldots, v_N) = \bigcup_i \left\{ 2 x\cdot v_i - |v_i|^2 \ge \ln y\right\}.
$$

Taking the complement, we obtain

\begin{lemma}
\label{lemma:min-gauss}
The value $P(v_1,\ldots, v_N)$ is maximized at a given configuration if the Gaussian measure of
$$
S_y(v_1,\ldots, v_N) = \bigcap_i \left\{ 2 x\cdot v_i - |v_i|^2 \le \ln y\right\}
$$
is minimized at the same point set $(v_1,\ldots, v_N)$ for any value of $y$.
\end{lemma}

An advantage of this approach is that the set $S_y(v_1,\ldots, v_N)$ is a (possibly unbounded) convex polyhedron. From the inequality in~\cite{sidak1967} we readily obtain:

\begin{theorem}
\label{theorem:sidak}
If we consider sets of $2N$ points ($N\le n$) in $\mathbb R^n$ of the form $\{v_i\}_{i=1}^N$ with prescribed $|v_i| =r_i$ then $P(v_1, - v_1, \ldots, v_N, - v_N)$ is maximized when all the vectors $v_i$ are orthogonal to each other.
\end{theorem}

\begin{proof}
In this case, the set $S_y$ is an intersection of several symmetric planks
$$
P_i = \left\{|(x, v_i)| \le \frac{|v_i|^2 - \ln y}{2} \right\},
$$
and the Gaussian measure of this intersection is minimized when all the stripes are perpendicular. This follows from the \v{S}id\'{a}k inequality~\cite{sidak1967}
$$
\mu(P_1\cap \dots\cap P_N) \ge \mu(P_1)\cdot \dots \cdot \mu(P_N),
$$
which becomes an equality in case when all the planks are perpendicular to each other. This perpendicularity is only possible when $N\le n$.
\end{proof}

We continue the discussion of such \emph{antipodal} configurations in Section~\ref{section:antipodal}. Similarly, Conjecture~\ref{conjecture:simplex} is therefore reduced to:

\begin{conjecture}
\label{conjecture:gausssimplex}
The Gaussian measure of a simplex $S$ containing a given ball $B_0(r)$ is minimized at the regular simplex with inscribed ball $B_0(r)$.
\end{conjecture}

Of course, by slicing and using the result~\cite{lft1953} about spherical caps we conclude that this conjecture holds true for $n=3$.

\begin{proof}[Reduction of Conjecture~\ref{conjecture:simplex} to Conjecture~\ref{conjecture:gausssimplex}]
In order to make such a reduction we have to establish that unbounded \emph{generalized simplices}, that is sets determined by $n+1$ linear equations in $\mathbb R^n$, are ruled out.

Call a generalized simplex \emph{essentially unbounded} if it contains an open cone. Equivalently, its outer normals of facets do not contain the origin in their convex hull. In Conjecture~\ref{conjecture:simplex} this corresponds to the case when the convex hull of $\{v_i\}$ does not contain the origin. Let $p$ be the closest to the origin point in this convex hull.

Assume that $p$ points to the ``north'' and let $E$ be the corresponding ``equator'' of $\mathbb S^{n-1}$. The point $p$ is a convex combination of some of $v_i$'s, without loss of generality let them be $v_1, \ldots, v_k$. Note that these $v_1,\ldots, v_k$ are at the same distance from $E$ and if we move them uniformly to $E$ (and keep other $v_i$'s fixed) then the pairwise distances between them increase. Moreover, any distance $|v_i - v_j|$ for $i\le k$ and $j > k$ also increases, in order to see this it is sufficient to consider the three-dimensional space spanned by $v_i, p, v_j$ and apply the elementary geometry.

Now we use the reduction of Conjecture~\ref{conjecture:simplex} to Conjecture~\ref{conjecture:balls} and analyze the volume of the union $\bigcup_{i=1}^{n+1} B_{v_i} (r)$ for every radius $r>0$. The continuous case of the Kneser--Poulsen conjecture established in~\cite{csi1998} asserts that for every $r>0$ the volume of such a union does not decrease when we move $v_1,\ldots, v_k$ to the equator. Hence the total value $P(v_1,\ldots, v_{n+1})$ does not decrease either. Now observe that at the end the origin will be in the convex hull of $v_i$'s.

Call a generalized simplex \emph{degenerate} if it is not essentially unbounded, but is still unbounded. Every degenerate simplex is a limit (in the topology given by the family of metrics $\dist_R(X, Y) = \dist_\textrm{Haus}( X\cap B_0(R), Y\cap B_0(R) )$, ($R>0$) of honest simplices; and it is easy to see that the Gaussian measure of a degenerate simplex will be the limit of the Gaussian measures of those honest simplices. So the inequality would follow, since we do not want it to be strict.
\end{proof}

After this, one may try to establish Conjecture~\ref{conjecture:gausssimplex} by taking the minimal example and studying its structure. There may be some difficulty: this minimal example may turn out to be degenerate. This could be avoided if we manage to prove the stronger version of Conjecture~\ref{conjecture:gausssimplex}:

\begin{conjecture}
\label{conjecture:radialsimplex}
Let $\mu$ be a radially symmetric measure with monotone decreasing positive density $\rho(r)$.
The value $\mu(S)$ over all simplices $S$ containing a given ball $B_0(r)$ is minimized at the regular simplex with inscribed ball $B_0(r)$.
\end{conjecture}

This conjecture can be attacked by the analysis of the minimizer because of

\begin{lemma}
\label{lemma:honestsimplex}
If the integral $\int_0^{+\infty} \rho(r)\; dr$ diverges then the minimum in Conjecture~\ref{conjecture:radialsimplex} is attained at an honest simplex.
\end{lemma}

\begin{proof}
Obviously, degenerate simplices have infinite measure in this case.
\end{proof}

\begin{lemma}
\label{lemma:minimizers}
Let $S_0$ be the regular simplex outscribed about $B_0(r)$. If for any measure $\mu$, satisfying the assumptions of Conjecture~\ref{conjecture:radialsimplex}, and any non-degenerate local minimizer (among honest simplices) $S$ of $\mu(S)$ under the constraint $S\supset B_0(r)$ we have $\mu(S) \ge \mu(S_0)$ then Conjecture~\ref{conjecture:radialsimplex} holds.
\end{lemma}

\begin{proof}
First, the assertion follows from Lemma~\ref{lemma:honestsimplex} if the integral $\int_0^{+\infty} \rho(r)\; dr$ diverges.

Let us consider the general case. Assume the contrary: suppose that $\mu(S)$ is minimized (over $S\supset B_0(r)$, of course) at a degenerate simplex $S$. If we have $\mu(S) < \mu(S_0)$, then we can approximate $S$ by an honest simplex $S'$ and still have $\mu(S') < \mu(S_0)$. Now we can change the density of $\mu$ so that it remains the same around $S_0$ and $S'$, and the integral $\int_0^{+\infty} \rho(r)\; dr$ diverges. But for the modified measure we have already shown that $\mu(S') \ge \mu(S_0)$.
\end{proof}

\section{Antipodal configurations in the plane}
\label{section:antipodal}

Let us focus on the case when the configuration is \emph{antipodal}, that is containing the vector $-v$ for every its vector $v$. Theorem~\ref{theorem:sidak} thus asserts that such a configuration becomes better if we make all the pairs $\pm v_i$ in it orthogonal to each other keeping their lengths.

But what about the lenghts? Let us analyze how the value $P(\pm v_1, \ldots, \pm v_N)$ behaves when the vectors $v_i$ are kept orthogonal to each other, but their lengths are allowed to vary. It would be nice if the maximization if $P(\pm v_1, \ldots, \pm v_N)$ for fixed $|v_1|^2+\dots + |v_N|^2$ happened at equal lengths $|v_1|= \dots = |v_N|$; but the example in~\cite{stein1994} is actually a counterexample to this naive conjecture. There only one pair of $\pm v_i$ was given the maximal possible length while all other pairs $\pm v_i$ were put to zero.

Fortunately, the naive conjecture holds in dimension 2:

\begin{theorem}
\label{theorem:energy-2d}
Under fixed $|v_1|^2 + |v_2|^2$ the maximum of $P(\pm v_1, \pm v_2)$ is given at equal and orthogonal to each other $v_1$ and $v_2$.
\end{theorem}

\begin{remark}
It will be clear from the proof that the conclusion remains true if consider, instead of the Gaussian measure, any measure with radially symmetric density.
\end{remark}

\begin{proof}
Of course, the picture is essentially planar. Let $a$ and $b$ be the lengths of $v_1$ and $v_2$ respectively, and let $a\leq b$ without loss of generality. Consider the Voronoi regions of the the four points in the plane. Let us move every Voronoi regions so that the center of it gets to the origin, see Figure~\ref{fig:movement}.

\begin{figure}[h]
\caption{Moving the Voronoi regions}
\label{fig:movement}
\begin{center}
\includegraphics{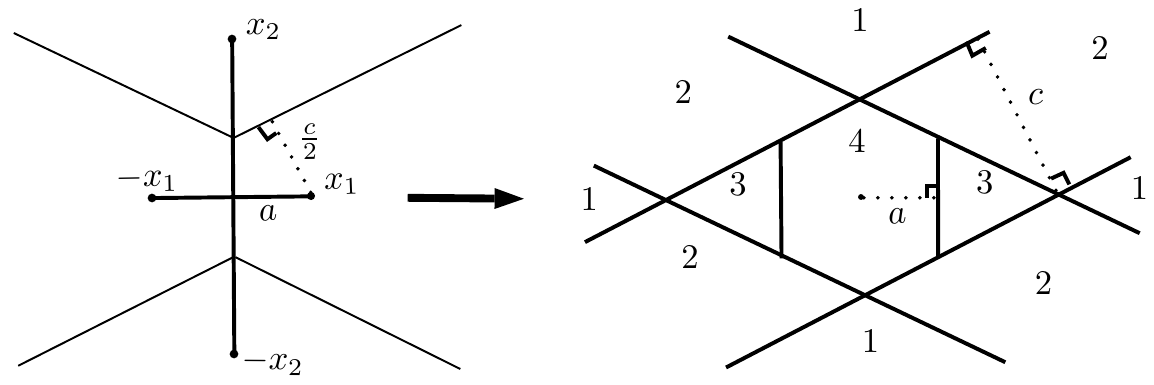}
\end{center}
\end{figure}

The numbers in the picture show how many times every area gets covered after the overlap of the moved Voronoi regions. Now all the measures in each of the Voronoi regions become the same Gaussian measure centered at the origin and we count it taking the overlap into account. This turns out to be the measure of the whole plane, plus two centrally symmetric strips of width $c = \sqrt{a^2 + b^2}$ each, plus the measure of the hexagon in the picture. Since the width of the strips $c$ does not depend on the choice of $a$ and $b$, their measure is also constant in fact.

Hence, for fixed $a^2 + b^2$, we maximize $P(\pm v_1, \pm v_2)$ if and only if we maximize the Gaussian measure of the hexagon. Let us look at the hexagon closer: It is obtained from the rhombus, which is the intersection of two strips of width $c$, by cutting off two corners. Let us give a geometric description of the cutting: Let  $O$ be the center of the hexagon and let $A, B, C$ (see Figure~\ref{fig:hexagon}) be its vertices. Since every two points of the configuration are symmetric with respect to the wall between their respective Voronoi regions, the point $N$, defined as symmetric to $O$ about $BC$ is on the straight line $AB$. Let $M$ be the base of the perpendicular from $O$ to $BC$.

\begin{figure}[h]
\caption{The hexagon}
\label{fig:hexagon}
\begin{center}
\includegraphics{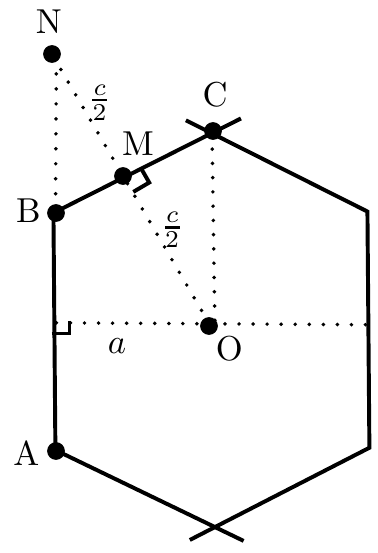}
\end{center}
\end{figure}

Since $AB$ and $OC$ are parallel, we obtain $\angle BNM = \angle MOC$. As was mentioned, $OM = MN$, $\angle BMN = \angle CMO$. Hence the triangles $\triangle BMN$ and $\triangle CMO$ are equal. Therefore $OM$ is the perpendicular bisector of $BC$ and $OB = CO$.

We conclude that the hexagon is characterized by the following properties: All its vertices are at equal distances from the origin; two antipodal pairs of its sides are at distance $\frac{c}{2}$ from the origin. In other words, the four sides touch the circle of radius  $\frac{c}{2}$ centered at $O$ at their respective midpoints.

Let us fix a direction in the plane and parameterize the hexagon by six parameters: Four angles for the sides that are $\frac{c}{2}$ from the origin and two shifts along the given direction for the remaining two sides, see Figure~\ref{fig:parametrization}.

\begin{figure}[h]
\caption{Hexagon's parameterization}
\label{fig:parametrization}
\begin{center}
\includegraphics{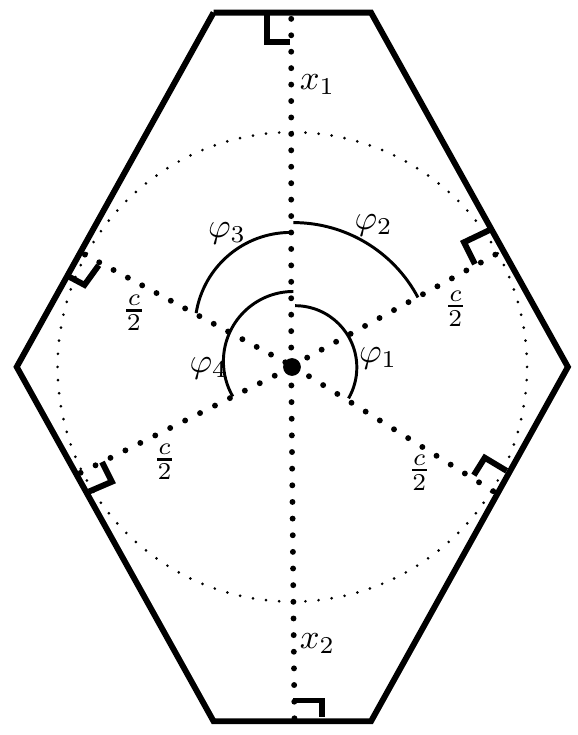}
\end{center}
\end{figure}

Let $F(\varphi_1, \varphi_2, \dots, x_2)$ be the Gaussian measure of such a hexagon; the center of the measure is also $O$.

When we change $a$ and $b$ keeping $c = \sqrt{a^2 + b^2}$ the hexagon vary. We may assume that the six parameters of the hexagon are all functions of $a$:
$$
F(a) = F(\varphi_1(a), \varphi_2(a), \dots, x_2(a)).
$$
Let us find the derivative:
$$
F'_a(a) = \sum\frac{\partial}{\partial \varphi_i}F(\varphi_1, \varphi_2, \dots, x_2)\varphi_i'(a) + \sum\frac{\partial}{\partial x_j}F(\varphi_1, \varphi_2, \dots, x_2)x_j'(a).
$$

\begin{lemma}
If the parameters correspond the hexagon in question (are expressed in $a$) then
$$
\forall i \in \{1,2,3,4\} \forall a\,\,\,\frac{\partial}{ \varphi_i} F(\varphi_1, \varphi_2, \dots, x_2) = 0.
$$
\end{lemma}

\begin{proof} When we change the angle $\varphi_i$ the corresponding side is rotated about the origin keeping in touch with the circle of radius $c$. Since the vertices of the hexagon are at the same distance from the origin and the Gaussian density is radially symmetric, then the mass center of the side (in this Gaussian density) is in its midpoint, which is the same as the touching with the circle.

If the touching point is rotated, say, with angular velocity $\omega$ then at start the velocity at a point $x$ of the side segment equals to $\omega|x|$ and is directed along $\overrightarrow{Ox}$ rotated by $\frac{\pi}{2}$. If we consider two such points symmetric to each other with respect to the midpoint of the side then we see that the densities are the same at those two points and the projections of their velocities onto the normal of the side sum to zero. Since in the linear term the measure changes by the integral over the side segment of the density multiplied by the normal component of the velocity, the total derivative of the measure with respect to the rotation turns out to be zero.
\end{proof}

Now we see that the partial derivatives of $F(\varphi_1, \varphi_2, \dots, x_2)$ in the angles are zero, and its partial derivatives in $x_1$ and $x_2$ are definitely non-negative and positive for $a < b$. Hence the measure increases when $x_1$ and $x_2$ increase. Since $x_1(a) = x_2(a) = a$ we have to increase $a$ until it becomes equal to $b$ (at this moment the picture changes). So $a = b$ is the optimal configuration.
\end{proof}

\section{Numerical results for antipodal configurations}
\label{section:antipodal-numerical}

\subsection{Formulas for the modified example of Steiner}

In Steiner's example~\cite{stein1994} one pair of antipodal vectors had nonzero length while the other pairs had length zero, that is all those vectors were the same at the origin. Let us generalize this as follows: $k$ pairs of vectors have the same length, while all other vectors are in the origin and their set is not empty. Let us write down an explicit formula for the probability in this case.

In order to calculate the function $P$ for the configuration we have to take every point in the configuration and its Voronoi region, integrate the Gaussian measure centered at this point over its Voronoi region, and then sum up the results over all the points.

The linear hull of our point is $k$-dimensional and their Voronoi regions in the ambient space are orthogonal products of $k$-dimensional Voronoi regions by the complementary linear subspace. The Gaussian measures also equal to the products of $k$-dimensional Gaussian measures by the Gaussian measure of the complement, which is $1$. Therefore it is sufficient to work in the $k$-dimensional linear hull of the points.

Now choose the coordinate frame so that our nonzero vectors are $\pm$ the basis vectors multiplied by $a$. Of course, it does not matter how many points of the configuration are put to the origin; so we assume there is one point at the origin. The Voronoi region of the origin is therefore the cube $[-\frac{a}{2}, \frac{a}{2}]^k$. Other Voronoi regions are the cones on the facets of the cube minus the cube itself.

Here we start to use the standard version of the Gaussian measure with density $\frac{1}{\sqrt{2\pi}} e^{-x^2/2}$ per dimension. This is needed to invoke the standard notation
$$
\Phi(x) = \frac{1}{\sqrt{2\pi}} \int_0^x e^{-t^2/2}\; dt.
$$
So the cube $[-\frac{a}{2}, \frac{a}{2}]^k$ has the Gaussian measure
$$
\left(\Phi(a/2) - \Phi(-a/2)\right)^n = (2\Phi(a/2) - 1)^k.
$$

Now consider the Voronoi region which is adjacent to the cube by its facet $x_1 = \frac{a}{2}$. When we intersect this region by the hyperplane $x_1 = b,\, b \geq \frac{a}{2}$, we obtain the cube $[-b, b]^{k-1}$in this hyperplane. Now we have to integrate the induced Gaussian measure of this cube from $x_1=a$ to $+\infty$. The induced Gaussian measure has center at the center of the cube and the additional factor
$$
\frac{1}{\sqrt{2\pi}}e^{-\frac{\rho^2}{2}},
$$
where $\rho$ is the distance from the center of the original Gaussian measure to the hyperplane. Since the center is at the axis point with $x_1 = a$, the induced measure of the section is
$$
\frac{1}{\sqrt{2\pi}}e^{-\frac{(b - a)^2}{2}}(2\Phi(b) - 1)^{k-1}.
$$

Eventually, the Gaussian measure of the Voronoi region is
$$
\int\limits_{\frac{a}{2}}^{+\infty}\frac{1}{\sqrt{2\pi}}e^{-\frac{(b - a)^2}{2}}(2\Phi(b) - 1)^{k-1}\; db.
$$

And the total value is
$$
P(\pm ae_1,\ldots,\pm ae_k, 0) =
2k\int\limits_{ \frac{a}{2}}^{+\infty}\frac{1}{\sqrt{2\pi}}e^{-\frac{(b - a)^2}{2}}(2\Phi(b) - 1)^{k-1}\; db + (2\Phi(a/2) - 1)^k.
$$

This formula is not very nice, but it allows us to make some numerical experiments.

\subsection{Numerical experiments with the modified example of Steiner}

We give the table where the function $P$ of the modified example of Steiner is calculated for different values of $k$ (the number of nonzero pairs) and the total energy $E$ in Table~\ref{tab:Steiner example}.

\begin{table}
\caption{Numerical results of the modified example of Steiner, $k$ increases from left to right}
\label{tab:Steiner example}
\begin{center}

\resizebox{\linewidth}{!}{
\begin{tabular}{c|cccccccccccccccccccc}

 &1.000 &2.000 &3.000 &4.000 &5.000 &6.000 &7.000 &8.000 &9.000 &10.000 &11.000 &12.000 &13.000 &14.000 &15.000 &16.000 &17.000 &18.000 &19.000 &20.000 \\
\hline
0.100 &1.178 &1.186 &1.180 &1.173 &1.166 &1.160 &1.154 &1.149 &1.144 &1.140 &1.136 &1.133 &1.130 &1.127 &1.124 &1.122 &1.119 &1.117 &1.115 &1.113 \\
0.200 &1.251 &1.267 &1.260 &1.250 &1.240 &1.231 &1.223 &1.215 &1.209 &1.202 &1.197 &1.192 &1.187 &1.183 &1.179 &1.176 &1.172 &1.169 &1.166 &1.163 \\
0.300 &1.307 &1.331 &1.324 &1.311 &1.299 &1.288 &1.277 &1.268 &1.260 &1.252 &1.245 &1.239 &1.233 &1.228 &1.223 &1.218 &1.214 &1.210 &1.206 &1.203 \\
0.400 &1.354 &1.385 &1.378 &1.365 &1.350 &1.337 &1.325 &1.314 &1.304 &1.295 &1.287 &1.280 &1.273 &1.267 &1.261 &1.255 &1.250 &1.246 &1.241 &1.237 \\
0.500 &1.395 &1.434 &1.428 &1.413 &1.396 &1.381 &1.368 &1.355 &1.344 &1.334 &1.325 &1.316 &1.309 &1.302 &1.295 &1.289 &1.283 &1.278 &1.273 &1.268 \\
0.600 &1.432 &1.479 &1.473 &1.457 &1.439 &1.423 &1.407 &1.394 &1.381 &1.370 &1.360 &1.350 &1.342 &1.334 &1.326 &1.320 &1.313 &1.307 &1.302 &1.296 \\
0.700 &1.465 &1.520 &1.515 &1.498 &1.479 &1.461 &1.445 &1.430 &1.416 &1.404 &1.392 &1.382 &1.373 &1.364 &1.356 &1.348 &1.341 &1.335 &1.329 &1.323 \\
0.800 &1.496 &1.559 &1.555 &1.537 &1.517 &1.498 &1.480 &1.464 &1.449 &1.435 &1.423 &1.412 &1.402 &1.393 &1.384 &1.376 &1.368 &1.361 &1.354 &1.348 \\
0.900 &1.525 &1.596 &1.593 &1.575 &1.553 &1.533 &1.513 &1.496 &1.480 &1.466 &1.453 &1.441 &1.430 &1.420 &1.410 &1.402 &1.394 &1.386 &1.379 &1.372 \\
1.000 &1.553 &1.631 &1.630 &1.610 &1.588 &1.566 &1.546 &1.527 &1.510 &1.495 &1.481 &1.469 &1.457 &1.446 &1.436 &1.427 &1.418 &1.410 &1.402 &1.395 \\
1.500 &1.670 &1.786 &1.793 &1.773 &1.746 &1.719 &1.693 &1.670 &1.648 &1.629 &1.611 &1.594 &1.579 &1.565 &1.552 &1.540 &1.529 &1.518 &1.509 &1.499 \\
2.000 &1.766 &1.919 &1.936 &1.916 &1.886 &1.855 &1.825 &1.797 &1.771 &1.748 &1.726 &1.706 &1.688 &1.671 &1.656 &1.641 &1.628 &1.615 &1.603 &1.592 \\
2.500 &1.848 &2.037 &2.066 &2.048 &2.015 &1.980 &1.946 &1.914 &1.885 &1.858 &1.833 &1.810 &1.789 &1.770 &1.752 &1.735 &1.719 &1.705 &1.691 &1.678 \\
3.000 &1.919 &2.144 &2.186 &2.170 &2.136 &2.098 &2.061 &2.025 &1.992 &1.962 &1.934 &1.908 &1.884 &1.862 &1.842 &1.823 &1.805 &1.789 &1.773 &1.758 \\
3.500 &1.983 &2.243 &2.298 &2.286 &2.251 &2.211 &2.170 &2.131 &2.095 &2.061 &2.030 &2.002 &1.975 &1.951 &1.928 &1.907 &1.887 &1.869 &1.852 &1.835 \\
4.000 &2.041 &2.334 &2.404 &2.396 &2.361 &2.318 &2.275 &2.233 &2.193 &2.157 &2.123 &2.092 &2.063 &2.036 &2.011 &1.988 &1.966 &1.946 &1.927 &1.909 \\
4.500 &2.093 &2.420 &2.505 &2.502 &2.467 &2.422 &2.376 &2.331 &2.289 &2.249 &2.213 &2.179 &2.148 &2.119 &2.092 &2.066 &2.043 &2.021 &2.000 &1.981 \\
5.000 &2.142 &2.501 &2.601 &2.603 &2.569 &2.523 &2.474 &2.427 &2.382 &2.339 &2.300 &2.264 &2.230 &2.199 &2.170 &2.143 &2.117 &2.093 &2.071 &2.050 \\
10.000 &2.473 &3.125 &3.392 &3.473 &3.466 &3.420 &3.359 &3.292 &3.225 &3.160 &3.098 &3.039 &2.985 &2.933 &2.885 &2.840 &2.797 &2.758 &2.720 &2.685 \\
15.000 &2.658 &3.548 &3.988 &4.172 &4.219 &4.195 &4.137 &4.063 &3.983 &3.902 &3.823 &3.747 &3.674 &3.605 &3.540 &3.479 &3.421 &3.367 &3.315 &3.266 \\
20.000 &2.772 &3.857 &4.462 &4.761 &4.877 &4.891 &4.849 &4.778 &4.693 &4.602 &4.510 &4.420 &4.332 &4.249 &4.169 &4.093 &4.021 &3.952 &3.888 &3.826 \\
25.000 &2.846 &4.090 &4.847 &5.265 &5.462 &5.525 &5.510 &5.451 &5.367 &5.272 &5.172 &5.071 &4.971 &4.874 &4.781 &4.692 &4.607 &4.525 &4.448 &4.375 \\
30.000 &2.894 &4.269 &5.165 &5.700 &5.984 &6.105 &6.126 &6.086 &6.011 &5.917 &5.812 &5.704 &5.595 &5.487 &5.382 &5.281 &5.184 &5.091 &5.002 &4.917 \\
35.000 &2.927 &4.409 &5.429 &6.077 &6.451 &6.638 &6.702 &6.689 &6.628 &6.539 &6.434 &6.322 &6.206 &6.089 &5.975 &5.863 &5.755 &5.651 &5.551 &5.455 \\
40.000 &2.949 &4.520 &5.650 &6.407 &6.872 &7.128 &7.240 &7.259 &7.218 &7.140 &7.039 &6.925 &6.805 &6.682 &6.559 &6.439 &6.321 &6.207 &6.097 &5.991 \\
45.000 &2.965 &4.608 &5.837 &6.695 &7.250 &7.578 &7.744 &7.801 &7.784 &7.720 &7.627 &7.515 &7.393 &7.266 &7.137 &7.009 &6.882 &6.759 &6.640 &6.524 \\
50.000 &2.975 &4.679 &5.995 &6.947 &7.591 &7.992 &8.216 &8.314 &8.326 &8.281 &8.198 &8.091 &7.970 &7.841 &7.707 &7.573 &7.440 &7.308 &7.180 &7.056 \\
55.000 &2.983 &4.736 &6.129 &7.170 &7.898 &8.374 &8.657 &8.800 &8.844 &8.821 &8.753 &8.654 &8.536 &8.407 &8.271 &8.132 &7.992 &7.854 &7.719 &7.586 \\
60.000 &2.988 &4.782 &6.244 &7.365 &8.175 &8.724 &9.069 &9.260 &9.340 &9.343 &9.291 &9.204 &9.091 &8.964 &8.827 &8.685 &8.541 &8.397 &8.255 &8.115 \\
65.000 &2.991 &4.820 &6.342 &7.538 &8.426 &9.047 &9.454 &9.695 &9.814 &9.845 &9.814 &9.739 &9.635 &9.512 &9.375 &9.232 &9.085 &8.936 &8.788 &8.642 \\
70.000 &2.994 &4.851 &6.426 &7.691 &8.652 &9.344 &9.813 &10.107 &10.267 &10.329 &10.320 &10.261 &10.168 &10.050 &9.916 &9.773 &9.623 &9.471 &9.319 &9.167 \\
75.000 &2.996 &4.876 &6.499 &7.826 &8.857 &9.618 &10.149 &10.495 &10.698 &10.794 &10.810 &10.770 &10.688 &10.579 &10.449 &10.307 &10.157 &10.003 &9.846 &9.690 \\
80.000 &2.997 &4.897 &6.562 &7.946 &9.042 &9.870 &10.463 &10.863 &11.110 &11.241 &11.285 &11.265 &11.198 &11.098 &10.974 &10.835 &10.685 &10.529 &10.370 &10.210 \\
85.000 &2.998 &4.914 &6.616 &8.052 &9.210 &10.102 &10.755 &11.209 &11.503 &11.671 &11.744 &11.746 &11.695 &11.606 &11.491 &11.356 &11.208 &11.052 &10.891 &10.727 \\
90.000 &2.998 &4.928 &6.663 &8.147 &9.363 &10.315 &11.028 &11.536 &11.876 &12.083 &12.187 &12.213 &12.181 &12.105 &11.999 &11.869 &11.725 &11.569 &11.408 &11.242 \\
95.000 &2.999 &4.940 &6.704 &8.232 &9.501 &10.512 &11.283 &11.845 &12.232 &12.479 &12.615 &12.667 &12.654 &12.594 &12.498 &12.376 &12.235 &12.082 &11.920 &11.753 \\
100.000 &2.999 &4.950 &6.740 &8.307 &9.627 &10.693 &11.521 &12.136 &12.570 &12.858 &13.028 &13.107 &13.116 &13.072 &12.988 &12.874 &12.739 &12.589 &12.429 &12.261 \\
120.000 &3.000 &4.975 &6.842 &8.538 &10.025 &11.287 &12.322 &13.140 &13.765 &14.221 &14.537 &14.737 &14.845 &14.880 &14.857 &14.790 &14.689 &14.561 &14.415 &14.254 \\
140.000 &3.000 &4.988 &6.903 &8.687 &10.301 &11.719 &12.929 &13.931 &14.736 &15.362 &15.831 &16.165 &16.389 &16.520 &16.576 &16.573 &16.522 &16.434 &16.316 &16.176 \\
160.000 &3.000 &4.994 &6.940 &8.786 &10.494 &12.035 &13.391 &14.553 &15.523 &16.311 &16.933 &17.408 &17.756 &17.996 &18.145 &18.220 &18.234 &18.198 &18.123 &18.015 \\
180.000 &3.000 &4.997 &6.963 &8.853 &10.630 &12.268 &13.743 &15.042 &16.159 &17.098 &17.868 &18.483 &18.960 &19.316 &19.568 &19.732 &19.822 &19.851 &19.828 &19.764 \\
200.000 &3.000 &4.998 &6.976 &8.898 &10.728 &12.441 &14.013 &15.427 &16.674 &17.749 &18.658 &19.409 &20.014 &20.490 &20.850 &21.111 &21.287 &21.389 &21.430 &21.420 \\
300.000 &3.000 &5.000 &6.998 &8.982 &10.937 &12.843 &14.685 &16.444 &18.106 &19.657 &21.085 &22.384 &23.548 &24.578 &25.477 &26.249 &26.904 &27.448 &27.893 &28.246 \\
400.000 &3.000 &5.000 &7.000 &8.997 &10.984 &12.952 &14.890 &16.787 &18.632 &20.412 &22.117 &23.737 &25.261 &26.682 &27.994 &29.195 &30.282 &31.258 &32.124 &32.885 \\
500.000 &3.000 &5.000 &7.000 &8.999 &10.996 &12.985 &14.960 &16.914 &18.840 &20.729 &22.574 &24.365 &26.095 &27.754 &29.335 &30.831 &32.236 &33.548 &34.761 &35.877 \\
600.000 &3.000 &5.000 &7.000 &9.000 &10.999 &12.995 &14.985 &16.964 &18.928 &20.870 &22.786 &24.669 &26.513 &28.310 &30.054 &31.738 &33.356 &34.901 &36.369 &37.756 \\
700.000 &3.000 &5.000 &7.000 &9.000 &11.000 &12.998 &14.994 &16.985 &18.967 &20.936 &22.890 &24.822 &26.730 &28.607 &30.449 &32.249 &34.002 &35.703 &37.346 &38.926 \\
800.000 &3.000 &5.000 &7.000 &9.000 &11.000 &12.999 &14.998 &16.993 &18.984 &20.968 &22.942 &24.903 &26.847 &28.771 &30.671 &32.543 &34.382 &36.184 &37.943 &39.655 \\
900.000 &3.000 &5.000 &7.000 &9.000 &11.000 &13.000 &14.999 &16.997 &18.993 &20.984 &22.969 &24.946 &26.912 &28.864 &30.800 &32.717 &34.611 &36.477 &38.314 &40.116 \\
1000.000 &3.000 &5.000 &7.000 &9.000 &11.000 &13.000 &15.000 &16.999 &18.997 &20.992 &22.983 &24.970 &26.949 &28.919 &30.877 &32.822 &34.751 &36.660 &38.548 &40.410 \\

\end{tabular}
}
\end{center}
\end{table}

The graph of $P$ as a function of $k$ and $E$ is given in Figure~\ref{fig:Steiner graph}. We note again that in practice one may want to vary the number $N$ of the code vectors. First, it makes sense to take $N=2k+1$ putting precisely one vector to the origin. Then one has to note that the actual amount of the transmitted information multiplied by the probability of correct transmission will be our number $P$ with the factor of $\frac{\log_2 N}{N}$.

\begin{figure}
\caption{Graph of results}
\label{fig:Steiner graph}
\hspace*{-2cm}
    \includegraphics[scale = 0.5]{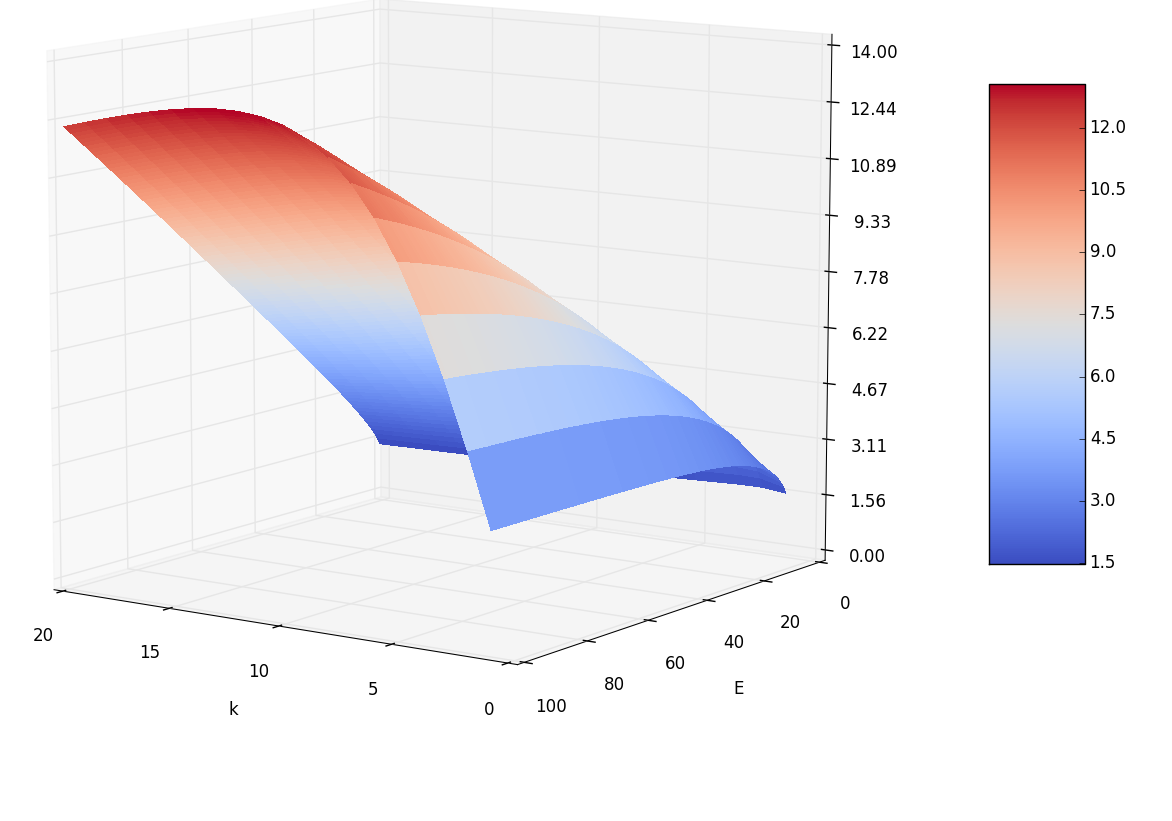}
\end{figure}

\subsection{Formulas for arbitrary orthogonal antipodal configuration}

In the more general case we have $k$ pairs of vectors $\pm v_i$, so that the lengths in $i$th pair are $a_i$. Again, we work in the linear hull of the configuration and consider the vectors of the configuration as proportional to the basis vectors.

Consider the hyperplane of $x_k=0$ and move it along the $k$th basis vector. Let the shifted hyperplane by $\{x_k=t\}$. In the intersection with this hyperplane, the Voronoi regions of the original points are the weighted Voronoi regions of their projections. The weights are $t^2$ for the first $k-1$ pairs, and the last pair is actually represented by one of the points projected to the origin in the hyperplane with weight $(a_{k} - t)^2$. When we subtract $t^2$ from all the weights then $2k-2$ points remain without weights, while the last one gets weight $a_k^2 - 2a_kt$. The latter Voronoi region is a parallelotope:
$$
\prod\limits_1^{k-1} \left[-\frac{a_i^2 + 2a_kt-a_k^2}{2a_i}, \frac{a_i^2 + 2a_kt-a_k^2}{2a_i}\right],
$$
if $a_i^2 + 2a_kt-a_k^2 > 0$ for all $i\neq k$.

In particular, in order to this latter Voronoi region to be nonempty, we need
$$
t > \frac{a_k^2 - \min\limits_{i\neq k}{a_i}^2}{2a_k},\, t > 0,
$$
that is $t > \frac{a_k^2 - \min{a_i}^2}{2a_k}$.

Write down the induced measure of this parallelotope:
$$
\prod\limits_1^{k-1} \left(2\Phi\left(\frac{a_i^2 + 2a_kt-a_k^2}{2a_i}\right) - 1\right),
$$
and integrate in $t$ to obtain the Gaussian measure of one of the points in $k$th pair:
$$
\int\limits_{\frac{a_k^2 - \min{a_i}^2}{2a_k}}^{+\infty} \frac{1}{\sqrt{2\pi}}e^{-\frac{(t - a_k)^2}{2}}\prod\limits_1^{k-1} \left(2\Phi\left(\frac{a_i^2 + 2a_kt-a_k^2}{2a_i}\right) - 1\right)\; dt.
$$

The sum of all such Gaussian measures equals:
$$
P(\pm a_1e_1, \ldots, \pm a_ke_k) = 2\sum\limits_{j = 1}^k\int\limits_{\frac{a_j^2 - \min{a_i}^2}{2a_j}}^{+\infty} \frac{1}{\sqrt{2\pi}}e^{-\frac{(t - a_j)^2}{2}}\prod\limits_{i \neq j} \left(2\Phi\left(\frac{a_i^2 + 2a_jt-a_j^2}{2a_i}\right) - 1\right)\; dt.
$$
This is the function of $(a_1,\ldots,a_k)$ we want to maximize.

\subsection{Numerical experiments for arbitrary orthogonal antipodal configuration}

Let us try to optimize the above function numerically. We are using the standard algorithm of \textit{Basin Hopping}. The results are given in Tables \ref{tab:k=3}, \ref{tab:k=4}, \ref{tab:k=5}, \ref{tab:k=6}.

\begin{table}[h]
\begin{center}
\caption{k = 3}
\label{tab:k=3}
E \begin{tabular}{c|ccc}
1.0&0.498&0.055&0.499\\
2.0&0.577&0.577&0.577\\
3.0&0.707&0.707&0.707\\
4.0&0.816&0.817&0.816\\
5.0&0.913&0.913&0.913\\
6.0&1.000&1.000&1.000\\
8.0&1.155&1.155&1.155\\
10.0&1.291&1.291&1.291\\
20.0&1.826&1.826&1.826\\
\end{tabular}
\end{center}
\end{table}

\begin{table}[h]
\begin{center}
\caption{k = 4}
\label{tab:k=4}
E \begin{tabular}{c|cccc}
2.0&0.578&0.000&0.577&0.577\\
3.0&0.707&0.707&0.026&0.707\\
4.0&0.816&0.816&0.052&0.816\\
5.0&0.791&0.791&0.791&0.791\\
6.0&0.866&0.866&0.866&0.866\\
8.0&1.000&1.000&1.000&1.000\\
10.0&1.118&1.118&1.118&1.118\\
20.0&1.581&1.581&1.581&1.581\\
\end{tabular}
\end{center}
\end{table}

\begin{table}[h]
\begin{center}
\caption{k = 5}
\label{tab:k=5}
E \begin{tabular}{c|ccccc}
1.0&0.001&0.000&0.055&0.498&0.499\\
2.0&0.000&0.577&0.578&0.577&0.000\\
3.0&0.707&0.707&0.000&0.026&0.707\\
4.0&0.816&0.002&0.816&0.816&0.052\\
5.0&0.791&0.010&0.791&0.791&0.791\\
6.0&0.017&0.866&0.866&0.866&0.866\\
8.0&1.000&0.035&1.000&1.000&1.000\\
10.0&0.067&1.118&1.118&1.118&1.118\\
12.0&1.095&1.095&1.095&1.096&1.095\\
14.0&1.183&1.183&1.183&1.183&1.183\\
16.0&1.265&1.265&1.265&1.265&1.265\\
20.0&1.414&1.414&1.414&1.414&1.414\\
\end{tabular}
\end{center}
\end{table}

\begin{table}[h]
\begin{center}
\caption{k = 6}
\label{tab:k=6}
E \begin{tabular}{c|cccccc}
2.0&0.011&0.000&0.577&0.577&0.577&0.000\\
3.0&0.707&0.707&0.000&0.707&0.000&0.026\\
4.0&0.816&0.816&0.052&0.000&0.816&0.000\\
5.0&0.011&0.790&0.789&0.792&0.000&0.791\\
6.0&0.000&0.016&0.866&0.866&0.866&0.866\\
8.0&1.000&1.000&0.035&1.000&1.000&0.000\\
10.0&0.012&1.000&1.000&1.000&1.000&1.000\\
12.0&0.019&1.095&1.096&1.095&1.095&1.095\\
14.0&1.183&1.183&1.183&1.183&0.031&1.183\\
16.0&1.265&1.265&0.048&1.265&1.265&1.265\\
18.0&1.225&1.225&1.225&1.225&1.225&1.225\\
19.0&1.258&1.258&1.258&1.258&1.258&1.258\\
20.0&1.291&1.291&1.291&1.291&1.291&1.291\\
21.0&1.323&1.323&1.323&1.323&1.323&1.323\\

\end{tabular}
\end{center}
\end{table}

It seems that for arbitrary dimension $k$ there exists a threshold of energy $E_0(k)$ such that for energy $E > E_0(k)$ the optimal configuration of $k$ antipodal pairs is the configuration with all equal lengths of the vectors.

\bibliography{../Bib/karasev}

\begin{thebibliography}{10}

\bibitem{balakri1961}
A.~V. Balakrishnan.
\newblock A contribution to the sphere-packing problem of communication theory.
\newblock {\em Journal of Mathematical Analysis and Applications}, 3:485--506,
  1961.

\bibitem{cover1987}
T.~M. Cover and B.~Gopinath.
\newblock {\em Open problems in Communication and Computation}.
\newblock Springer Verlag, New York, 1987.

\bibitem{csi1998}
B.~Csik\'os.
\newblock On the volume of the union of balls.
\newblock {\em Discrete \& Computational Geometry}, 20(4):449--461, 1998.

\bibitem{dunbridge1965}
B.~Dunbridge.
\newblock {\em Optimal signal design for the coherent Gaussian channel}.
\newblock PhD thesis, University of Southern California, Los Angeles, 1965.

\bibitem{dunbridge1967}
B.~Dunbridge.
\newblock Asymmetric signal design for the coherent {G}aussian channel.
\newblock {\em IEEE Transactions on Information Theory}, 13:422--431, 1967.

\bibitem{farber1968}
S.~M. Farber.
\newblock {\em On the signal selection problem for phase coherent and
  incoherent communication channels}.
\newblock PhD thesis, California Institute of Technology, 1968.

\bibitem{lft1953}
L.~Fejes~T\'oth.
\newblock {\em Lagerungen in der Ebene auf der Kugel und im Raum}.
\newblock Springer Verlag, Berlin, 1953.

\bibitem{ls1966}
H.~J. Landau and D.~Slepian.
\newblock On the optimality of the regular simplex code.
\newblock {\em Bell System Technical Journal}, 45(8):1247--1272, 1966.

\bibitem{stein1994}
M.~Steiner.
\newblock The strong simplex conjecture is false.
\newblock {\em IEEE Transactions on Information Theory}, 40(3):721--731, 1994.

\bibitem{sidak1967}
Z.~\v{S}id\'{a}k.
\newblock Rectangular confidence regions for the means of multivariate normal
  distributions.
\newblock {\em Journal of the American Statistical Association},
  62(318):626--633, 1967.

\end{thebibliography}
\bibliographystyle{abbrv}
\end{document}